\documentclass[12pt]{amsart}
\usepackage{amssymb,amscd,amsthm,amsmath,color, xcolor, xypic}
\usepackage{fullpage}
\newcommand{\PP}{\mathbb{P}}
\newcommand{\OO}{\mathcal{O}}

\newcommand{\FF}{\mathcal{F}}

\newcommand{\cB}{\mathcal{B}}
\newcommand{\GG}{\mathcal{G}}

\newcommand{\ZZ}{\mathbb{Z}}

\newcommand{\bG}{\mathbb{G}}
\newcommand{\cU}{\mathcal{U}}

\newcommand{\PGL}{\operatorname{PGL}}

\newcommand{\DD}{\mathcal{D}}

\newcommand{\codim}{\operatorname{codim}}

\newcommand{\Sym}{\operatorname{Sym}}

\theoremstyle{plain}
\newtheorem{lemma}{Lemma}[section]
\newtheorem*{theorem*}{Theorem}
\newtheorem*{lemma*}{Lemma}
\newtheorem*{proposition*}{Proposition}
\newtheorem*{conjecture*}{Conjecture}
\newtheorem*{corollary*}{Corollary}
\newtheorem*{problem*}{Problem}
\newtheorem{theorem}[lemma]{Theorem}

\newtheorem{corollary}[lemma]{Corollary}
\newtheorem{proposition}[lemma]{Proposition}

\newtheorem{fact}[lemma]{Fact}

\theoremstyle{definition}
\newtheorem{definition}[lemma]{Definition}
\newtheorem{example}[lemma]{Example}
\newtheorem{remark}[lemma]{Remark}

\newtheorem{notation}[lemma]{Notation}

\begin{document}

\title{Clustered families and applications to Lang-type conjectures}
\author[I. Coskun]{Izzet Coskun}
\address{Department of Mathematics, Statistics and CS \\University of Illinois at Chicago, Chicago, IL 60607}
\email{coskun@math.uic.edu}

\author[E. Riedl]{Eric Riedl}
\address{Department of Mathematics \\ University of Notre Dame, Notre Dame, IN 46556}
\email{ebriedl@nd.edu}

\thanks{During the preparation of this article the first author was partially supported by the   NSF FRG grant  DMS 1664296 and  the second author was partially supported  by the NSF CAREER grant DMS-1945944.}

\begin{abstract}
We introduce and classify $1$-clustered families of linear spaces in the Grassmannian $\bG(k-1,n)$ and give applications to Lang-type conjectures. Let $X \subset \PP^n$ be a very general hypersurface of degree $d$. Let $Z_L$ be the locus of points contained in a line of $X$.  Let $Z_2$ be the locus of points on $X$ that are swept out by lines that meet $X$ in at most $2$ points.
We prove that:
\begin{itemize}
\item  If $d \geq \frac{3n+2}{2}$, then $X$ is algebraically hyperbolic outside $Z_L$. 
\item If $d \geq \frac{3n}{2}$, $X$ contains lines but  no other rational curves. 
\item If  $d \geq \frac{3n+3}{2}$, then  the only points on $X$ that are rationally Chow zero  equivalent to points other than themselves are contained in $Z_2$. 
\item If $d \geq \frac{3n+2}{2}$ and a relative Green-Griffiths-Lang Conjecture holds, then the exceptional locus for $X$  is contained in $Z_2$.
 \end{itemize}

\end{abstract}

\maketitle

\section{Introduction}
Inspired by Faltings' proof of the Mordell Conjecture,  Lang made a series of conjectures relating rational points, hyperbolicity, abelian varieties and rational curves. For example, for a variety $X$ of general type, Lang conjectures \cite{Lang} that there exists a proper subvariety $Z \subset X$ such that
\begin{enumerate}
\item The images of nonconstant maps from rational curves and abelian varieties into $X$ are contained in $Z$.
\item The images of nonconstant entire curves are contained in $Z$.
\item The complement of $Z$ is Kobayashi hyperbolic.
\end{enumerate}
Furthermore, Lang predicts that these geometric conditions control the arithmetic of $X$ and conjectures that if $X$ is defined over a number field $K$ and $L$ is an algebraic extension of $K$, then $X \setminus Z(L)$ is finite. In this paper, we prove algebraic analogues of Lang-type conjectures with an explicit description of $Z$ for hypersurfaces of sufficiently high degree.

A projective variety $Y$ is {\em algebraically hyperbolic} if there exists $\epsilon >0$ such that any reduced, connected curve $C \subset Y$ of geometric genus $g(C)$ satisfies
\begin{equation}\label{eqn-genus}
2g(C) - 2 \geq \epsilon \deg(C).
\end{equation}
Demailly conjectures that for projective varieties algebraic hyperbolicity is equivalent to Kobayashi hyperbolicity \cite{Demaillynew}. 
Let $Z$ be a proper subvariety of $Y$. We say that $Y$ is {\em algebraically hyperbolic outside of $Z$} if inequality  (\ref{eqn-genus}) holds for any curve $C$ not contained in  $Z$.

Let $X$ be a very general  hypersurface of degree $d$ in $\PP^n$, $n \geq 3$. Let $Z_L$ denote the locus of points contained in a line of $X$. Let $Z_i$ denote the locus in $X$ swept out by lines meeting $X$ in at most $i$ points. 
By results of Ein \cite{Ein, Ein2}, Pacienza \cite{Pacienza}, Voisin \cite{Voisin, Voisincorrection} and the authors  \cite{CoskunRiedlhyperbolicity}, a very general hypersurface of degree $d \geq 2n-2 + \max(0, 4-n) $ is algebraically hyperbolic. When $d \leq 2n-3$, every hypersurface of degree $d$ contains lines, so $X$ cannot be algebraically hyperbolic. When the degree of the hypersurface is at least $n+2$, the Lang conjectures predict  the existence of a proper subvariety $Z$ such that $X$ is algebraically  hyperbolic outside of $Z$. In this paper, we develop a new technique for approaching Lang-type conjectures. 
The first application of this technique is finding a proper subvariety containing the exceptional set for algebraic hyperbolicity when the degree of the hypersurface is sufficiently large.

\medskip

\noindent{\bf Theorem \ref{thm-AH}.}
{\it If $d \geq \frac{3n+2}{ 2}$,  then any
curve not lying in $Z_L$ satisfies $2g(C) - 2 \geq \deg C$, where $g(C)$ is the geometric genus of $C$. In particular, $X$ is algebraically hyperbolic outside of $Z_L$.}
\medskip

Pacienza \cite{Pacienza2} proves a similar statement under the assumptiont that $d \geq \max \{ \frac{7n-6}{4}, \frac{3n}{2} \}$. We give a simpler proof that improves the bound with the optimal coefficient $\frac{3n}{2}$, since if $d \leq \frac{3n-2}{2}$, then $X$ contains conics, which necessarily lie outside $Z_L$ by \cite{RiedlYang1}.

Points on $X$ that lie on rational curves are part of Lang exceptional set. Such points are rationally Chow-0 equivalent. Hence, the set of points that are rationally Chow-0 equivalent to another point give another perspective in studying the Lang exceptional sets.  
Chen, Lewis and Sheng \cite{CLS}, inspired by the work of Voisin \cite{Voisin3, Voisin, Voisincorrection} conjecture that for a very general hypersurface $X$ and any point $p \in X$, then the dimension of the space $R_{\PP^1,X, p}$ of points of $X$ rationally Chow-0 equivalent to $p$ (other than $p$ itself) is at most $2n-d-1$. Chen, Lewis and Sheng \cite{CLS} prove the conjecture when the expected dimension is negative and Riedl and Yang \cite{RiedlYang} prove the rest of the conjecture. Our second application is to characterize the locus of rationally Chow-0 equivalent points on a very general hypersurface of sufficiently large degree.

\medskip

\noindent{\bf Theorem \ref{thm-Chow0}.}
{\it Let  $X$ be a very general hypersurface in $\PP^n$ of degree $d$.  
\begin{enumerate}
\item Let $k$ be a positive integer. If $d \geq \frac{3n+1-k}{2}$, then the only points of $X$ rationally equivalent to a $k$-dimensional family of points other than themselves are those that lie in $Z_1$. 
\item If $d \geq \frac{3n}{2}$, then $X$ contains lines but no other rational curves. 
\item If $d \geq \frac{3n+3}{2}$, then any point on  $X$ rationally equivalent to another point of $X$ lies in $Z_2$.
\end{enumerate}
}

\medskip

Theorem \ref{thm-Chow0} (1) is sharp. Theorem \ref{thm-Chow0} (2) generalizes a theorem of Clemens and Ran \cite{ClemensRan} who proved that when $d \geq \frac{3n+1}{2}$ all rational curves on $X$ are contained in $Z_L$.  Theorem \ref{thm-Chow0} (2) is sharp when $n$ is even. When $n$ is odd, we do not know whether a very general hypersurface of degree $\frac{3n-1}{2}$ can contain rational curves other than lines.

Our final application is to the exceptional set in the Green-Griffiths-Lang Conjecture. The conjecture predicts that if a variety $Y$ is of general type, then the images of all nonconstant entire curves are contained in a proper algebraic subvariety.

\begin{theorem*} [see Theorem \ref{thm-GGL}]
 If $d \geq \frac{3n+2}{2}$ and a relative version of the  Green-Griffiths-Lang Conjecture holds, then the exceptional locus for $X$  is contained in $Z_2$.
\end{theorem*}

Our approach for solving these problems involves a careful analysis of linear sections of very general hypersurfaces. We develop the study of $\ell$-clustered families of subspaces in the Grassmannian. Let $B \subset \bG(k-1, n)$ denote an irreducible family of $(k-1)$-dimensional projective linear spaces in $\PP^n$. Assume that the codimension of $B$ is $\epsilon >0$. Let $C \subset \bG(k,n)$ denote the family of $k$-dimensional projective linear spaces consisting of those linear spaces that contain a member of $B$. We call $C$ the {\em containing family} of $B$. By results of \cite{RiedlYang}, the codimension of $C$ in $\bG(k,n)$ is at most $\epsilon-1$. The family $B$ is {\em $\ell$-clustered} if the codimension of $C$ in $\bG(k,n)$ is $\epsilon -\ell$. In this paper, we classify $1$-clustered families in $\bG(k-1,n)$.
\medskip

\noindent {\bf Theorem \ref{thm-kPlanesCodim2}.}
{\it Let $B$ be an irreducible $1$-clustered family of $(k-1)$-dimensional linear spaces in $\PP^n$. Then the codimension $\epsilon$ of $B$ in $\bG(k-1,n)$ is at most $n-k+1$. If $\epsilon \geq 2$, then there is some irreducible subvariety $Z \subset \PP^n$ of dimension $n-k+1-\epsilon$ such that $B$ is the set of $(k-1)$-dimensional linear spaces intersecting $Z$.}
\medskip

More generally, we characterize the possible cohomology classes of $\ell$-clustered families and completely classify the extremal $\ell$-clustered families.
\medskip

\noindent {\bf Theorem \ref{thm-cohomologyclass}.}
{\it Let $\ell < k < n$ be integers. Let $B \subset \bG(k-1, n)$ be an irreducible $\ell$-clustered family of $(k-1)$-dimensional linear spaces. Let $[B]= \sum_{\lambda} a_{\lambda} \sigma_{\lambda}$ be the cohomology class of $B$ expressed in the Schubert basis.  
\begin{enumerate}
\item  Then any partition $\lambda$ occurring in $[B]$ with $a_{\lambda} \not=0$ has $\lambda_i =0$ for $i >\ell$. 
\item In particular, the codimension $\epsilon$ of $B$ is at most $\ell(n-k+1)$.
\item If $\epsilon = \ell (n-k+1)$, then $B$ parameterizes $(k-1)$-dimensional linear spaces that contain a fixed $\PP^{\ell-1}$.
\end{enumerate}}
\medskip

The classification of $1$-clustered families yields the stated applications via our main technical result Theorem \ref{thm-MainIncidenceThm}.

\subsection*{Organization of the paper} In \S \ref{sec-Prelim}, we collect facts about Grassmannians and rigidity of Schubert cycles. In \S \ref{sec-Clustered}, we introduce and classify $1$-clustered families. This forms the main technical tool of the paper. In \S \ref{sec-Main}, we use the classification of $1$-clustered families to prove our main theorem on families of hypersurfaces. We then deduce applications to rationally Chow zero equivalent points, rational curves and algebraic hyperbolicity.

\subsection*{Acknowledgments} We would like to thank Lawrence Ein, Joe Harris, Mihai P\u{a}un and David Yang for fruitful conversations on the subject of this paper. We are especially grateful to Xi Chen for bringing the problems discussed in this paper to our attention.

\section{Preliminaries}\label{sec-Prelim}

\subsection{Preliminaries on the Grassmannian} In this subsection, we  briefly recall preliminary facts about the Grassmannian and its cohomology. We refer the reader to \cite{CoskunLR, CoskunRigid, CoskunImanga} for more details.

Let $\bG(k,n)$ denote the Grassmannian that parameterizes $k$-dimensional projective linear subspaces of $\PP^n = \PP V$. The Grassmannian $\bG(k,n)$ is a projective variety of dimension $(k+1)(n-k)$ and embeds in $\PP( \bigwedge^{k+1} V) \cong \PP^{\binom{n+1}{k+1} -1}$ under the Pl\"{u}cker embedding. The cohomology of $\bG(k,n)$ is generated by Schubert classes. Let $\lambda$ be a partition with $k+1$ parts that satisfy
$$n-k \geq \lambda_1 \geq \cdots \geq \lambda_{k+1}\geq 0.$$ Fix a complete flag $F_{\bullet} : 0 \subset F_1 \subset \cdots \subset F_{n+1} = V$ on $V$, where $F_i$ is an $i$-dimensional subspace of $V$. The Schubert variety $\Sigma_{\lambda_{\bullet}}(F_{\bullet})$ is defined by 
$$\Sigma_{\lambda}(F_{\bullet}) := \{ \PP W \in \bG(k,n) \ | \ \dim (W \cap F_{n-k+i - \lambda_i}) \geq i, 1 \leq i \leq k+1 \}.$$
The class $\sigma_{\lambda}$ of $\Sigma_{\lambda}(F_{\bullet})$ does not the depend on the flag. The Schubert classes $\sigma_{\lambda}$ give an additive $\ZZ$-basis for the cohomology ring $H^*(\bG(k,n), \ZZ)$ as $\lambda$ varies over all admissible partitions. The cohomology class of any subvariety of $\bG(k,n)$ is a nonnegative $\ZZ$-linear combination of Schubert classes. In particular, the product of two Schubert classes can be expressed as a nonnegative linear combination of other Schubert classes
$$\sigma_{\lambda} \cdot \sigma_{\mu} = \sum_{\nu} c_{\lambda, \mu}^{\nu} \sigma_{\nu},$$ where $c_{\lambda, \mu}^{\nu} \geq 0$ are the Littlewood-Richardson coefficients. 

\begin{notation}
Given a partition $\lambda$ with $\lambda_{1} \not= n-k$, let $\lambda^h$ be the partition with $\lambda^h_i = \lambda_i +1$ for every $1 \leq i \leq k+1$. Similarly, given a partition with $\lambda_{k+1} = 0$, let $\lambda^p$ be the partition with $\lambda^p_1 = n-k$ and $\lambda^p_i = \lambda_{i-1}$ for $1< i \leq k+1$.
\end{notation}
We will need the following basic facts.
\begin{fact}\label{fact-intersection} Let $\sigma_{\lambda}$ and $\sigma_{\nu}$ be Schubert classes in the cohomology of $\bG(k,n)$.
\begin{enumerate}
\item Then $\sigma_{\lambda} \cdot \sigma_{\mu} \not= 0$ if and only if $\mu_i \leq n-k - \lambda_{k+2 -i}$ for all $1 \leq i \leq k+1$.
\item Let $\mu$ be the partition with $\mu_i =1$ for $1 \leq i \leq k+1$. Let $\lambda$ be a partition with $\lambda_1 \not= n-k$. Then  $\sigma_{\lambda} \cdot \sigma_{\mu} = \sigma_{\lambda^h}$.
\item Let $\mu$ be the partition with $\mu_1 = n-k$ and $\mu_i =0$ for $1< i \leq k+1$. Let $\lambda$ be a partition with $\lambda_{k+1}=0$. Then $\sigma_{\lambda} \cdot \sigma_{\mu} = \sigma_{\lambda^p}$.
\end{enumerate}
\end{fact}

A Schubert class $\sigma_{\lambda}$ in $\bG(k,n)$ is called {\em rigid} if any closed subvariety representing $\sigma_{\lambda}$ is a Schubert variety. A Schubert class $\sigma_{\lambda}$ is called {\em multi-rigid} if any closed subvariety of $\bG(k,n)$  representing $m \sigma_{\lambda}$ is a union of $m$ Schubert varieties. A complete classification of rigid and multi-rigid Schubert classes in $\bG(k,n)$ is known thanks to the work of Hong, Robles, The and the first author (see \cite{CoskunRigid, CoskunImanga, CoskunRobles, Hong, Hong2, RoblesThe}). Let $\lambda$ be the partition with $\lambda_i=n-k$ for $1 \leq i \leq j$ and $\lambda_i = 0$ for $j < i \leq k+1$. Then $\sigma_{\lambda}$ is multi-rigid in $\bG(k,n)$.

\section{Classification of $\ell$-clustered families}\label{sec-Clustered}
In this section, we classify $1$-clustered families and prove results about $\ell$-clustered families in general.

\begin{notation}
Let $B \subset \bG(k-1,n)$ be an irreducible family of $(k-1)$-dimensional linear subspaces of $\PP^n$. Let $\epsilon>0$ be the codimension of $B$ in $\bG(k-1,n)$. Let $C \subset \bG(k,n)$ be the family of all $k$-dimensional linear spaces that contain a member of $B$. We refer to $C$ as the \emph{containing family} of $B$. Let $I_{B,C}$ denote the incidence correspondence 
$$I_{B,C} = \{ (b, c) | b \in B, c \in C, b \subset c \} \subset B \times C$$  and let $\pi_1$ and $\pi_2$ be the two projections from $I_{B,C}$ to $B$ and  $C$, respectively.
\end{notation}

By assumption, the fiber of $\pi_1$ over $b \in B$ consists of all $k$-dimensional linear subspaces of $\PP^n$ containing $b$, hence is isomorphic to $\PP^{n-k}$. Consequently, $I_{B,C}$ is irreducible of dimension $$\dim I_{B,C} = (k+1)(n-k)+k - \epsilon.$$  We say that the family $B$ is {\em $\ell$-clustered} if the general fiber of $\pi_2$ has dimension $(k-\ell)$. In other words, $B$ is $\ell$-clustered if the general $c \in C$ contains an $(k-\ell)$-dimensional family of $(k-1)$-dimensional linear subspaces contained in $B$. If $B$ is $\ell$-clustered, then $C= \pi_2(I_{B,C})$ is an irreducible variety of dimension $(k+1)(n-k) - \epsilon + \ell$. We conclude that the codimension of $C$ is $\epsilon - \ell$. For future reference, we state this in the following lemma.

\begin{lemma}\label{lem-cluster}
Let $B \subset \bG(k-1,n)$ be an irreducible family of codimension $\epsilon$. The family $B$ is $\ell$-clustered if and only if the codimension of the containing family $C$ is $\epsilon -\ell$.
\end{lemma}

Note that if $B \not= \bG(k-1, n)$, then $\ell \geq 1$ and the codimension of $C$ is at most $\epsilon -1$. The purpose of this section is to classify $1$-clustered families $B$. We will also make some remarks about $j$-clustered families. We begin with a few illustrative examples. 

\begin{example}
Let $B \subset \bG(k-1,n)$ be the Schubert variety parameterizing $(k-1)$-dimensional projective linear spaces that contain a fixed $\PP^{j-1}$. Then $B$ has codimension $j(n-k+1)$. The containing family $C \subset \bG(k,n)$ parameterizes $k$-dimensional projective linear spaces that contain the same $\PP^{j-1}$ and has codimension $j(n-k)$. In particular, by Lemma \ref{lem-cluster}, $B$ is $j$-clustered.
\end{example}

\begin{example}\label{ex-main}
More generally, let $\Gamma \subset \bG(j-1,n)$ be a subvariety and assume that a $k$-dimensional linear space containing a member of $\Gamma$ contains only finitely many elements of $\Gamma$. Let $B \subset \bG(k-1, n)$ be the variety that parameterizes $(k-1)$-dimensional linear spaces that contain a member of $\Gamma$.  The containing family $C$ consists of $k$-dimensional linear spaces that contain a member of $\Gamma$. Hence, the general member  of $C$ contains a $(k-j)$-dimensional family of members of $B$. Consequently,  $B$ is a $j$-clustered family.
\end{example}

We give a final example to show that not all examples of $j$-clustered families arise as in Example \ref{ex-main}.

\begin{example}\label{ex-nasty}
Let $H$ be a general hyperplane section of $\bG(1,3) \subset \PP^5$ which is not in the dual variety of $\bG(1,3)$. Then $H$ is not a Schubert variety and there is no line which intersects every  line parametrized by $H$. Fix a linear subspace $\Lambda=\PP^3 \subset \PP^4$ and fix a point $p$ not contained in $\Lambda$.  Let $B \subset \bG(2,4)$ be the family of planes obtained by taking the span of a line in $\Lambda$ parameterized by $H$ and the point $p$. Note that $B$ has codimension $3$, has cohomology class  $\sigma_{2,1,0}$, and is not a Schubert variety. The containing family $C$ consists of linear spaces $\PP^3$ that contain the point $p$.  Hence,  by Lemma \ref{lem-cluster}, $B$ is $2$-clustered.  However, there does not exist a $1$-parameter family $\Gamma$ of lines such that $B$ is the family of planes containing a member of $\Gamma$. Each of the lines in $\Gamma$ would have to contain the point $p$, hence the intersection of the lines in $\Gamma$ with $\Lambda$ would be a curve, which by degree considerations has to be a line. However, by our choice of $H$, there does not exist such a line. 
\end{example}

We now classify $1$-clustered families in $\bG(k-1,n)$. The proof will be by induction on $n$. The next proposition provides the base case of the induction.

\begin{proposition}\label{prop-basecase}
Let $B \subset \bG(k-1, k+1)$ be an irreducible family of codimension $\epsilon \geq 2$. Then $B$ is $1$-clustered if and only if $\epsilon=2$ and $B$ is the Schubert variety of $(k-1)$-dimensional linear spaces containing a fixed point.
\end{proposition}
\begin{proof}
If $B$ is the family of $(k-1)$-dimensional linear spaces in $\PP^{k+1}$ containing a fixed point $p$, then the containing family $C$ is the family of $k$-dimensional linear spaces containing $p$. The codimension of $B$ in $\bG(k-1,k+1)$ is $2$ and the codimension of $C$ in $\bG(k,k+1)$ is $1$. Hence, by Lemma \ref{lem-cluster}, $B$ is $1$-clustered. 

Conversely, suppose $B$ is $1$-clustered.  Fix a general point $ c \in C$. Then $c$ contains a $(k-1)$-dimensional family $B'$ of elements of $B$. For each element $b' \in B'$, every $k$-dimensional linear space $c'$ containing $b'$ must be a member of $C$. The $k$-dimensional linear spaces containing $b'$ is a line in $\bG(k,k+1) = \PP^{(k+1) \vee}$. Hence, $C$ contains a $(k-1)$-dimensional family of lines all containing the point $c \in C$. These lines are distinct and  sweep out a $k$-dimensional subvariety of $C$. Consequently, $C$ has dimension at least $k$, or equivalently, codimension at most $1$.  Since $B$ is $1$-clustered and has codimension $\epsilon \geq 2$, we conclude that  $\epsilon = 2$ and $C$ has codimension exactly $1$. Since $C$ is irreducible and $k$-dimensional, $C$ must consist of all the $k$-dimensional linear spaces that contain an element of $B'$. In particular, through every point $c'$ of $C$, there is a line connecting $c$ and $c'$. Since $c$ was a general point, we conclude that there is a line between any two general points of $C$. Hence, $C$ is a linear space of codimension one in $\PP^{(k+1) \vee}$. All such linear spaces consist of $k$-dimensional linear spaces containing a fixed point $p$. 

Now it follows that every member of $B$ also contains the point $p$. If there were an element $b \in B$ not containing $p$, then there would exists a linear space $c$ containing $b$ and not $p$. The linear space $c$ would be a member of $C$ contradicting our description of $C$. Since the locus of $(k-1)$-dimensional linear spaces containing the point $p$ has codimension 2 in $\bG(k-1, k+1)$, $B$ must contain all these linear spaces. 
\end{proof}

We now use induction to classify the $1$-clustered families of $(k-1)$-dimensional linear spaces in every $\PP^n$.

\begin{theorem}
\label{thm-kPlanesCodim2}
Let $B$ be an irreducible $1$-clustered family of $(k-1)$-dimensional linear spaces in $\PP^n$. Then the codimension $\epsilon$ of $B$ in $\bG(k-1,n)$ is at most $n-k+1$. If $\epsilon \geq 2$, then there is some irreducible subvariety $Z \subset \PP^n$ of dimension $n-k+1-\epsilon$ such that $B$ is the set of $(k-1)$-dimensional linear spaces intersecting $Z$.
\end{theorem}
\begin{proof}
Let $C$ denote the containing family of $B$. Given a hyperplane $H$, let $B_H$ be the set of planes of $B$ lying in $H$.  Let $C_H$ denote the containing family of $B_H$ in $H$. We make the simple observation that if $B$ is $1$-clustered and $B_H$ is nonempty, then $B_H$ is also $1$-clustered. We will prove the theorem by induction on $n$. The case $n=k+1$ is Proposition \ref{prop-basecase}, so we may assume $n \geq k+2$. There are two cases depending on whether $B_H$ is empty or not.

Assume that the cohomology class of $B$ is $\sum_{\lambda} a_{\lambda} \sigma_{\lambda}$. If $\epsilon < n-k+1$, then $\lambda_1 < n-k+1$ for every partition $\lambda$ occurring in the class of $B$ with nonzero coefficient $a_{\lambda}$. The class of the locus of linear spaces that are contained in $H$ is the Schubert class $\sigma_{\mu}$, where $\mu_i =1$ for $1 \leq i \leq k$. Consequently, by Facts \ref{fact-intersection}  (1) and (2), the cohomology class of $B_H$ is nonzero. Hence, in this case $B_H$ is necessarily nonempty. Similarly, by Facts \ref{fact-intersection} (1) and (2), if $\epsilon \geq n-k+1$, $B_H$ is empty if and only if every $\lambda$ for which $a_{\lambda} \not=0$ has $\lambda_1 = n-k+1$. Now we analyze the two cases.

\textbf{Case 1:} For a general hyperplane $H$, $B_H$ is nonempty.  In this case, each component of $B_H$ has codimension $\epsilon$ in $\bG(k-1, n-1)$, and it follows by our induction hypothesis that $\epsilon \leq n-k$. Moreover, by our induction hypothesis, $B_H$ is the set of $(k-1)$-planes meeting some (possibly reducible) subvariety $Z_H$ of dimension $n-k-\epsilon$ in $H$. Take a general pencil of hyperplanes $H_{t,u} : = \{ tH+uH'\}$ containing $H$. The same analysis applies to the general member $H_{t,u}$ in this pencil and the members of $B$ lying in $H_{t,u}$ meet a variety $Z_{t,u}$. Moreover, the varieties $Z_{t,u}$ vary algebraically and trace out a subvariety $Z \subset \PP^n$ of dimension $n-k+1-\epsilon$. Let $U$ be a $k$-dimensional linear space that intersects $Z$ at a point $p \in Z$. Then $p \in H_{t,u}$ for some hyperplane in the pencil and $U$ contains an element of $B_{H_{t,u}}$, hence is in $C$. We conclude that every $k$-dimensional linear space intersecting $Z$ must be in $C$. 

The locus of $k$-dimensional linear spaces intersecting $Z$ has codimension $\epsilon -1$. Since $C$ is irreducible and has codimension $\epsilon -1$, we conclude that $Z$ must be irreducible and $C$ must equal the $k$-dimensional linear spaces intersecting $Z$. It follows that every $(k-1)$-dimensional linear space parameterized by $B$ also intersects $Z$. The locus of $(k-1)$-dimensional linear spaces intersecting $Z$ is irreducible of codimension $\epsilon$, hence $B$ must be the family of $(k-1)$-dimensional linear spaces intersecting $Z$. This concludes the induction in this case.

\textbf{Case 2:} For a general hyperplane $H$, $B_H$ is empty. In this case, we have $\epsilon \geq n-k+1$. Let $H$ be a hyperplane that is general among those that contain some element of $B$. Then the codimension of $B_H$ in $\bG(k-1,n-1)$ is at most $\epsilon-1$. Furthermore, $B_H$ is $1$-clustered in $\bG(k-1,n-1)$. By induction, the codimension of $B_H$ in $\bG(k-1,n-1)$ is at most $n-k$.  If the codimension were less than $n-k$, then the general hyperplane would contain a member of $B_H$, contrary to our assumption. We conclude that the codimension is $n-k$. Hence, by induction $B_H$ consists of $(k-1)$-dimensional linear spaces in $H$ that contain some fixed points $p_1, \dots, p_r$. Then $C$ contains the $k$-dimensional linear spaces that contain $p_i$. Hence, the codimension of $C$ is at most $n-k$. Since $B$ is $1$-clustered, we conclude that $\epsilon = n-k+1$. Since $C$ is irreducible, in fact $r=1$ and $C$ must consist of linear spaces that contain a fixed point $p$. Now every element of $B$ must also contain this point $p$. Since $B$ has codimension $n-k+1$, $B$ must consist of all the $(k-1)$-dimensional linear spaces containing $p$. This concludes the induction in this case.

We thus conclude that any $1$-clustered family in $\bG(k-1, n)$ whose codimension $\epsilon \geq 2$ is the set of $(k-1)$ dimensional linear spaces that intersect a subvariety of dimension $n-k+1-\epsilon$.
\end{proof}

The proof of Case 2 shows the following generalization of Proposition \ref{prop-basecase} holds.

\begin{corollary}
Let $B$ be an irreducible  family of $(k-1)$-dimensional linear spaces in $\PP^n$ of codimension $n-k+1$. Then $B$ is $1$-clustered if and only if $B$ parameterizes linear spaces containing a fixed point. 
\end{corollary}

As Example \ref{ex-nasty} indicates, Theorem \ref{thm-kPlanesCodim2} does not have an easy generalization. However, we can restrict the possible cohomology classes of $j$-clustered families and classify the extremal cases.

\begin{theorem}\label{thm-cohomologyclass}
Let $j < k < n$ be integers. Let $B \subset \bG(k-1, n)$ be an irreducible $j$-clustered family of $(k-1)$-dimensional linear spaces. Let $[B]= \sum_{\lambda} a_{\lambda} \sigma_{\lambda}$ be the cohomology class of $B$.  
\begin{enumerate}
\item  Then any partition $\lambda$ occurring in $[B]$ with $a_{\lambda} \not=0$ has $\lambda_i =0$ for $i >j$. 
\item In particular, the codimension $\epsilon$ of $B$ is at most $j(n-k+1)$.
\item If $\epsilon = j(n-k+1)$, then $B$ parameterizes $(k-1)$-dimensional linear spaces that contain a fixed $\PP^{j-1}$.
\end{enumerate}
\end{theorem}

\begin{proof}
Let $\lambda$ be a partition such that $\sigma_{\lambda}$ occurs in the class of $B$ with nonzero coefficient $a_{\lambda}$. Assume that $\lambda_{\ell} \not= 0$ and $\lambda_{\ell+1}=0$. Let $\lambda^*$ be the dual partition defined by $$\lambda^*_i = n-k+1 - \lambda_{k+1-i}.$$ Then, by Kleiman's transversality theorem, a general Schubert variety with class $\sigma_{\lambda^*}$ intersects $B$ in $a_{\lambda}$ points. 
More explicitly, fix a general partial flag $$F_{1+ \lambda_k} \subset \cdots \subset F_{i+\lambda_{k-i+1}} \subset \cdots \subset F_{k+\lambda_1}.$$ Then there are $a_{\lambda}$ members of the family $B$ that also intersects $\PP F_{i+\lambda_{k-i+1}}$ in an $(i-1)$-dimensional projective linear space for $1 \leq i \leq k$. 

Now define a new partition $\mu$ for $\bG(k,n)$. Let $$\mu_i = \begin{cases} n-k  &  \mbox{for} \ 1 \leq i \leq k-\ell+1 \\ \lambda^*_{i-1} & \mbox{for} \ k-\ell +1< i \leq k+1\end{cases}.$$ Observe that $\mu$ is an admissible partition for $\bG(k,n)$ and for $i > k-\ell +1$, 
$$n-k+i - \mu_i = i -1 - \lambda_{k+2 - i}$$ A general Schubert variety with class $\sigma_{\mu}$ parameterizes $k$-dimensional projective linear spaces that contain a fixed projective space $\PP F_{k-\ell+1}$ of dimension $k-\ell$ and intersect linear spaces $\PP F_{i-1+\lambda_{k-i+2}}$ in an $i$-dimensional projective linear space for $k-\ell +1 < i \leq k+1$. The key observation is  that the intersection of a general Schubert variety $\Sigma_{\mu}$ with the containing family $C$ is nonempty. There are finitely many elements of $B$ containing the fixed linear space $\PP F_{k-\ell}$ and intersecting the linear spaces $\PP F_{i+\lambda_{k-i+1}}$ in an $(i-1)$-dimensional linear space for $k-\ell < i \leq k$. These $(k-1)$-dimensional linear spaces together with the fixed linear space $\PP F_{k-\ell +1}$ span a $k$-dimensional linear space in the containing family $C$.

The codimension of $\mu$ is $$(n-k) (k-\ell +1) + \sum_{i=\ell+1}^k \lambda_i^*.$$ Consequently, the codimension of $C$ is at most $$\codim (C) \leq (n-k)\ell - \sum_{i=\ell+1}^k \lambda_i^*.$$ On the other hand, the codimension of $B$ is $$\codim(B) = (n-k+1)\ell -  \sum_{i=\ell+1}^k \lambda_i^*.$$ Since $B$ is $j$-clustered, we conclude that $j =\codim(B) - \codim(C) \geq \ell$. This shows that if $\sigma_{\lambda}$ occurs in the class of a $j$-clustered subvariety $B$ with nonzero coefficient, then the number of nonzero entries $\ell$ in $\lambda$ can be at most $j$. This concludes the proof of (1).

Since the only nonzero entries in $\lambda$ can be $\lambda_i$ with $1 \leq i \leq j$ and $\lambda_i$ is at most $n-k+1$, we see that the codimension $$|\lambda| = \sum_{i=1}^{k} \lambda_i \leq j (n-k+1).$$ This concludes the proof of (2).

Let $\lambda$ be the partition such that $\lambda_i = n-k+1$ for $1 \leq i \leq j$ and $\lambda_i = 0$ for $j< i\leq k$. By parts (1) and (2), the cohomology class of $B$ has to be a multiple of $\sigma_{\lambda}$. This is a multi rigid cohomology class, in other words, the only representatives of multiples of this class are unions of Schubert varieties. Since $B$ is irreducible, it must consist of a single Schubert variety. The corresponding Schubert variety parameterizes $(k-1)$-dimensional linear spaces that contain a fixed $\PP^{j-1}$.  This concludes the proof of (3). 
\end{proof}

\begin{remark}
It is possible to give an elementary proof of Theorem \ref{thm-cohomologyclass} (3) without relying on classification of multi rigid cohomology classes. We will  give the argument. 

\begin{lemma}\label{lem-jClusteredBaseCase}
Let $B \subset \bG(k-1,n)$ be a $j$-clustered irreducible family of codimension $\epsilon = 2j$. Then $B$ consists of all of the $(k-1)$-dimensional linear spaced that contain a fixed $\PP^{j-1}$. 
\end{lemma}
\begin{proof}
Let $C$ be the containing family of $B$. The general member $c \in C$ contains a $(k-j)$-dimensional family of members of $B$. Each $b \in B$ with $b \subset c$ the $k$-dimensional linear spaces containing $b$ is a line contained in $C$.  Hence, $\dim C \geq k-j+1$, or equivalently $\codim C \leq j$. Since $B$ is $j$-clustered of codimension $2j$, we conclude that $\codim C =j$ and $C$ must consist of  the union of lines containing $c$.  Since $c$ was a general point, we conclude that $C$ is a linear space of codimension $j$. Hence, $C$ parameterizes $k$-dimensional linear spaces that contain a fixed $\PP^{j-1}$. Every element of $B$ must also contain this $\PP^{j-1}$. The locus of $(k-1)$-dimensional linear spaces containing a fixed $\PP^{j-1}$ is an irreducible Schubert variety of codimension $2j$. We conclude that $B$ must equal this Schubert variety. 
\end{proof}

Lemma \ref{lem-jClusteredBaseCase} proves Theorem \ref{thm-cohomologyclass} (3) when $n=k+1$. Suppose that (3) holds by induction up to $n-1$.  Let $B \subset \bG(k,n)$ be an irreducible, $j$-clustered subvariety of codimension $j(n-k+1)$. Let $H$ be a general hyperplane among those that contain a member of $B$. Let $B_H$ be the members of $B$ that are contained in $H$. Observe that $B_H$ is also $j$-clustered. If the codimension of $B_H$ is $j(n-k)$, then, by induction on $n$, $B_H$ consists of all $(k-1)$-dimensional linear spaces that contain a fixed $\Lambda \cong \PP^{j-1}$. Then $C$ contains all the $k$-dimensional linear spaces that contain $\Lambda$. Since $C$ has codimension $j(n-k)$, $C$ must equal the locus of $k$-dimensional linear spaces that contain $\Lambda$. Then $B$ must also equal the locus of $(k-1)$-dimensional linear spaces that contain $\Lambda$ as desired.

There remains to show that $B_H$ has codimension $j(n-k)$ in $\bG(k-1, n-1)$. Let $Z$ be the space of hyperplanes in $\PP^n$ that contain an element of $B$. It suffices to show that $\dim (Z) \geq n-j$. By Theorem \ref{thm-cohomologyclass} (2) and Fact \ref{fact-intersection} (1), there exists elements of $B$ that contain a fixed $\Lambda \cong \PP^{k-1-j}$.  Given this element $b$ of $B$, there is a $\PP^{n-k}$ dimensional family of hyperplanes containing $b$. Hence, the codimension $k-j$ linear space of hyperplanes containing $\Lambda$ intersects $Z$ in a subvariety of dimension at least $n-k$. Therefore, the dimension of $Z$ is at least $n-j$, as desired. This concludes the induction and gives an elementary proof of Theorem \ref{thm-cohomologyclass} (3).
\end{remark}

\section{Incidence Argument}\label{sec-Main}

In this section, we introduce the notion of towers of induced subvarieties and classify towers where the codimension of linear sections increases by one each time using Theorem \ref{thm-kPlanesCodim2}. This classification stated in Theorem \ref{thm-MainIncidenceThm} and Proposition \ref{prop-d-1d-2cases} is the main tool for our applications to rational curves and Chow equivalence. 

\begin{notation}
Let $\cU_{n,d}$ denote the universal hypersurface of degree $d$ in $\PP^n$ parameterizing pairs $(p,X)$, where $X \subset \PP^n$ is a hypersurface of degree $d$ and $p\in X$ is a point. We call a pair $(p,X) \in \cU_{n,d}$ a  {\em pointed hypersurface}.  

We will study linear sections of pointed hypersurfaces. It is convenient to view all these linear sections in a fixed projective space. For this purpose, we  introduce  the notion of a parameterized linear section. For $r < n$, a {\em parametrized linear space} of dimension $r$ is a linear map $f: \PP^r \to \PP^n$. When $r=1$, we simply say {\em parameterized line}. Let $p \in \PP^n$ be a point. Let $\GG_p(k,n)$ denote the {\em space of parameterized $k$-dimensional linear spaces in $\PP^n$ through $p$}. Let $(p,X)$ be a pointed hypersurface in $\PP^n$ and let $f$ be a parameterized $k$-plane in $\PP^n$. Assume the image of $f$ contains $p$ but is not contained in $X$. Then  the {\em parameterized linear section} of $(p,X)$  by $f$ is $(f^{-1}(p), f^{-1}(X))$.

Given a $\PGL_{r+1}$-invariant subvariety $\cB_{r,d} \subset \cU_{r,d}$,   the {\em tower of induced varieties} $\cB_{n,d}$ for $n \geq r$ are the subvarieties $\cB_{n,d} \subset \cU_{n,d}$ consisting of pointed hypersurfaces such that some parameterized linear section lies in $\cB_{r,d}$. 
\end{notation}

We first construct a large dimensional pointed hypersurface $(p,Y)$ in $\PP^N$ that contains  every pointed hypersurface $(p,X)$ in $\PP^n$ as a parameterized linear section.

\begin{lemma}\label{lem-allparameterized}
For every pair of positive integers $n,d$, there exists an $N=N(n,d)$ and a  pointed hypersurface $(p,Y)$ in $\PP^N$ such that every member of $\cU_{n,d}$ is a parameterized linear section of $(p,Y)$. Furthermore, for every member $(p,X)$ of $\cU_{n,d}$, there exists a parameterized hyperplane  $\Lambda$ in $\GG_p(n,N)$ such that the induced map on tangent spaces $$d\phi_n: T_{\Lambda} \GG_p(n,N) \to T_{(p,X)} \cU_{n,d}$$ is surjective.
\end{lemma}

\begin{proof} First, we show that if $N$ is sufficiently large, then there exists a pointed hypersurface $(p,Y)$ in $\PP^N$ such that the parameterized linear sections of $(p,Y)$ dominate $\cU_{n,d}$ and the map $d \phi_n$ is surjective at any given point. We will then use Noetherian induction to complete the proof of the lemma.

 Let $(p,Y)$ be a pointed hypersurface in $\PP^N$ and let $f: \PP^n \to \PP^N$ be a parameterized linear section. Let $I_p \subset \OO_{\PP^n}$ be the ideal sheaf of the point $p$ in $\PP^n$ and let $T_{\PP^N}$ denote the tangent bundle of $\PP^N$. We then have the following commutative diagram 
$$ \xymatrix{  & 0 & 0 & & \\ 0 \ar[r] & f^*(T_{\PP^N} (- \log Y))\otimes I_p  \ar[u]\ar[r]  & f^*(T_{\PP^N} )\otimes I_p \ar[u] \ar[r]^{\phi} &  f^*(\OO_Y(Y)) \otimes I_p \ar[r] & 0  \\  0 \ar[r]  & E \otimes I_p \ar[u] \ar[r] & I_p(1)^{N+1} \ar[u]\ar[r]^{\psi} & f^*(\OO_Y(Y)) \otimes I_p \ar[u]_= \ar[r] & 0 \\ & I_p \ar[u] \ar[r]_= & I_p \ar[u] & \\ & 0 \ar[u] & 0 \ar[u] & }, $$
where the first row is the pullback by $f$ of the standard exact sequence expressing the log tangent sheaf of $Y$ tensored by $I_p$ and the second column is the pullback of the Euler sequence by $f$ tensored by $I_p$ and $E$ is the pullback of the kernel of the map $\OO(1)^{N+1} \to \OO_Y(Y)$ by $f$. By basic deformation theory, to show that the parameterized linear sections dominate $\cU_{n,d}$ near $f$, we need to show that the map induced by $\phi$
$$H^0(\PP^n, f^*(T_{\PP^N} )\otimes I_p) \to H^0( \PP^n, f^*(\OO_Y(Y)) \otimes I_p)$$ is surjective. By the diagram, it suffices to show that $H^1(\PP^n, f^*(T_{\PP^N} (- \log Y))\otimes I_p )=0$. Since $I_p$ has no cohomology, this is equivalent to showing that $H^1(\PP^n, E \otimes I_p) =0$. Hence, it suffices to show that the map induced by $\psi$ 
$$H^0(\PP^n, I_p(1)^{N+1}) \to H^0(\PP^n, f^*(\OO_Y(Y))\otimes I_p)$$
is surjective. 

Assume that $Y$ is defined by the polynomial $g$. This map is given by the $(N+1) \times 1$ matrix consisting of the pullback of the partial derivatives of $g$ by  $f$. It is now clear that we can pick $Y$ so that this map is surjective if $N$ is sufficiently large. Suppose for a given $Y$ the map is not surjective. We can then pick polynomials of degree $d-1$, say $g_{N+1}, \dots, g_{N+\ell}$, so  that  the map becomes surjective. If we replace the pointed hypersurface $(p, Y)$ by the hypersurface $(p, Y')$ in $\PP^{N+ \ell}$ defined by the polynomial $g + \sum_{i=1}^{\ell} x_{N+i} g_{N+i}$, then the corresponding map is surjective.  

To complete the proof, suppose $(p,X)$ does not occur as a parameterized linear section of $(p,Y)$, then we can find a hypersurface $(p,Y_1)$ that has both $(p,X)$ and $(p,Y)$ as hyperplane sections and is locally surjective around $(p,X)$. Consequently, by Noetherian induction, we can find a single hypersurface that works for all pointed hypersurfaces in $\cU_{n,d}$. 

\end{proof}

\begin{remark}
Using the notation from Lemma \ref{lem-allparameterized},  given a subvariety $\cB_n \subset \cU_{n,d}$, the codimension of $\cB_n$ in $\cU_{n,d}$ around a point where $d\phi_n$ is surjective is the same as the codimension of $\phi_n^{-1}(\cB_n)$ in $\GG_p(n,N)$. 
\end{remark}

\begin{definition}
Fix a pointed hypersurface $(p,Y)$ in $\PP^N$. A parameterized $n$-plane $\Lambda$ is \emph{$(n,m)$-submersive} if for  each $r$ with $m \geq r \geq n$, there exists a parameterized $r$-plane $\Lambda_r$ containing $\Lambda$ such that  $\phi_r$ is a submersion near $\Lambda_r$.
\end{definition}

\begin{remark}
\label{rem-submersiveHyp}
A straightforward adaptation of the proof of Lemma \ref{lem-allparameterized} shows that given integers $d$, $n$ and $m$, there exists a pointed hypersurface $(p,Y)$ in $\PP^N$ such that for every $(p,X)$ in $\cU_{n,d}$, there is a parameterized $n$-plane $\Lambda$ in $\PP^N$ with $\phi_n(\Lambda) = (p,X)$ that is $(n,m)$-submersive.
\end{remark}

We next prove two easy lemmas that will be used in the proof of the main theorem of this section. 

\begin{lemma}
\label{lem-glueAlongLine}
Let $X_1 \subset  \PP^{n_1}$ and $X_2 \subset \PP^{n_2}$  be degree $d$ hypersurfaces and let $\ell_1$ and $\ell_2$ be parameterized lines in $\PP^{n_1}$ and $\PP^{n_2}$, respectively. Suppose that $\ell_1^* X_1 = \ell_2^* X_2$ . Then there exists an integer $N$, a degree $d$ hypersurface $Y$ in $\PP^N$, a parameterized line $\ell$ in $\PP^N$ and parameterized $n_1$ and $n_2$-dimensional linear spaces  $\Lambda_1$ and $\Lambda_2$ such that
\begin{enumerate}
\item $\ell^* Y = \ell_1^* X_1 = \ell_2^* X_2$
\item $\Lambda_1^* Y = X_1$ and 
\item $\Lambda_2^* Y = X_2$.
\end{enumerate}
\end{lemma} 

\begin{proof}
Choose coordinates  $x_0, \dots, x_{n_1}$  on $\PP^{n_1}$ and  $y_0, \dots, y_{n_2}$ on $\PP^{n_2}$ such that the lines $\ell_1 = [s:t:0:\dots : 0]$ and $\ell_2 = [s:t:0: \dots : 0]$ are defined by the vanishing of $x_i$ for $i \geq 2$ and $y_i$ for $i \geq 2$, respectively. Suppose that in these coordinates $X_1$ is the vanishing locus of $f_1$ and $X_2$ is the vanishing locus of $f_2$.  Since $\ell_1^{*} X_1 = \ell_2^{*} X_2$, we can scale the polynomials $f_1$ and $f_2$ so that $ f_1 = g(x_0, x_1)+h_1(x_0, . . . , x_{n_1} )$, where every monomial  occurring in $h_1$ is divisible by $x_2,\dots, x_{n_1}$ and $f_2 = g(y_0,y_1)+h_2(y_0,\dots, y_{n_2})$, where every monomial in $ h_2$ is divisible by $y_2,\dots, y_{n_2}$. Define $f$ by $$f =g(z_0,z_1)+h_1(z_2,...,z_{n_1})+h_2(z_{n_1 +1}, \dots, z_{n_1 + n_2 -2}).$$ Then set $Y =V(f)$, let $\ell$ be  the line defined by $z_i=0$ for $2 \leq i \leq n_1+n_2 -2$,  let $\Lambda_1$ be the linear space defined by $z_i=0$ for $n_1+1\leq i \leq n_1+n_2 -2$ and let $\Lambda_2$ be defined by $z_i=0$ for $2\leq i \leq n_1$. Then $Y$ satisfies the conclusions of the lemma.
\end{proof}

\begin{lemma} \label{lem-clusteredNondecreasing}
Suppose $B \subset \bG(k-1,n)$ has codimension at least 2. Then if $B$ is 1-clustered, its covering set $C$ is also 1-clustered.
\end{lemma}
\begin{proof}
Let $B \subset \bG(k-1,n)$ have codimension $\epsilon \geq 2$. Since $B$ is $1$-clustered,   Theorem \ref{thm-kPlanesCodim2} implies that $B$ parameterizes $(k-1)$-dimensional linear spaces that intersect a subvariety $Z \subset \PP^n$. Furthermore, the covering family $C$ parameterizes the $k$-dimensional linear spaces that intersect $Z$ as well and has codimension $\epsilon -1$ in $\bG(k,n)$. Let $D \subset \bG(k+1,n)$ be the covering family of $C$. Then $D$ parameterizes the $(k+1)$-dimensional linear spaces that intersect $Z$ and has codimension $\epsilon -2$ in $\bG(k+1,n)$. By Lemma \ref{lem-cluster}, we conclude that $C$ is $1$-clustered.

\end{proof}

The main theorem of this section is the following.

\begin{theorem}
\label{thm-MainIncidenceThm}
Let $\cB_{r,d}$ be an integral $\PGL_{r+1}$-invariant subvariety of $\cU_{r,d}$. For $n \geq r$, let $\cB_{n,d}$ denote the tower of induced subvarieties. Assume that for some $m \geq r$,  $\cB_{m,d}$ has codimension at least $1$ in $\cU_{m,d}$. Then either
\begin{enumerate}
\item  for $r \leq n \leq m$,  $\cB_{n,d}$ has codimension at least $2(m-n)+1$ in $\cU_{n,d}$; or
\item  there exists a $\PGL_2$-invariant family $\FF \subset \cU_{1,d}$ such that $\cB_{m,d}$ is in the closure of the tower of varieties induced by $\FF$; or
\item $\cB_{m,d}$ is the space of pairs $(p,X)$ with $p$ contained in a line $\ell$ lying in $X$.
 \end{enumerate}
\end{theorem}
\begin{proof}
Let $(p,Y)$ be a pointed hypersurface of degree $d$ in $\PP^N$ constructed in Remark \ref{rem-submersiveHyp} such that every member of $\cU_{m,d}$ is a parameterized linear section of $(p,Y)$ and furthermore, each fiber of $\phi_r$ contains an $(r,m)$-submersive point. Hence, for $r \leq n \leq m$, we have a surjective rational map $\phi_n: \GG_p(n,N) \dashrightarrow \cU_{n,d}$ from the space of parameterized $n$-dimensional linear spaces to degree $d$ pointed hypersurfaces in $\PP^n$. Let $B_n$ be the union of all components of $\phi_n^{-1}(\cB_{n,d})$ containing an $(n,m)$-submersive parameterized $n$-plane. We know that $B_n$ is nonempty by the construction of $Y$. By the definition of $(n,m)$-submersive, it follows that $B_{n+1}$ is a covering set for $B_n$ and each $B_n$ is equidimensional.

Since the map $\phi_m$ is surjective and $\cB_{m,d}$ has codimension at least 1, $B_m$ has codimension at least 1 in  $\GG_p(m,N)$. Suppose for all  $r < k \leq m$, $$\codim (B_{k-1} \subset \GG_p(k-1, N)) \geq \codim (B_{k} \subset \GG_p(k, N)) + 2,$$ then for $r \leq n \leq m$ we have $$\codim(B_{n} \subset \GG_p(n, N)) \geq 2(m-n)+1.$$
Since $B_{n}$ is a union of components of  $\phi_n^{-1}(\cB_{n,d})$ and dominates $\cB_{n,d}$, we conclude that $$\codim(\cB_{n,d} \subset \cU_{n,d}) \geq 2(m-n) +1.$$

Since the $\cB_{n,d}$ are $\PGL_{n+1}$-invariant, $B_n$ is a torsor over a family $\Gamma_n \subset \bG(n-1, N-1)$, viewed as $n$-planes in $\PP^N$ passing through $p$. Furthermore, $\Gamma_n$ by definition is the containing family of $\Gamma_{n-1} \subset \bG(n-2, N-1)$. 

By Lemma \ref{lem-clusteredNondecreasing}, if $\codim (\Gamma_{m-1} \subset \bG(m-2, N-1)) \geq \codim (\Gamma_{m} \subset \bG(m-1, N-1)) + 2,$ then
$\codim (\Gamma_{k} \subset \bG(k-1, N-1)) \geq \codim (\Gamma_{k+1} \subset \bG(k, N-1)) + 2$
  for all $k \leq m-1$. Thus, it remains to consider the case where 
$$\codim (\Gamma_{m-1} \subset \bG(m-2, N-1)) = \codim (\Gamma_{m} \subset \bG(m-1, N-1)) + 1.$$

By assumption, $\Gamma_{m-1} \subset \bG(m-2, N-1)$ is $1$-clustered. By Theorem \ref{thm-kPlanesCodim2}, there exists a variety $Z \subset \PP^{N-1}$ such that $\Gamma_{m-1}$ consists of the $(m-2)$-dimensional linear spaces that intersect $Z$. Hence, $B_{m-1}$ consists of all parameterizations of $(m-1)$-dimensional linear spaces that contain a member of the family of lines through $p$ and a point of  $Z$.

Suppose that none of the lines parameterized by $Z$ lie in $Y$. We prove that $\cB_{m-1,d}$ is a piece of the tower of varieties induced by a $\PGL_2$-invariant family $\FF \subset \cU_{1,d}$. Let $\psi: \GG_p(1, N) \to \PP^{N-1}$ be the morphism that maps a parameterized line to its image, where $\PP^{N-1}$ is viewed as the family of lines passing through $p$. Let $\FF = \phi_1(\psi^{-1}(Z)) \subset \cU_{1,d}$. It is clear that $\FF$ is $\PGL_2$-invariant.

By definition, we see that for each $\ell \in \FF$, there must be a parameterized line $\ell_1$ in $\PP^N$ lying in a parameterized $(m-1)$-plane $\Lambda_1 \in B_{m-1}$ with $\ell_1^* Y = \ell$. If $\cB_{m-1,d}$ is not in the tower of  varieties induced by $\FF$, then for some $\ell$, we can also find a parameterized line $\ell_2$ lying in a parameterized $(m-1)$-plane $\Lambda_2 \notin B_{m-1}$ with $\ell_2^*Y = \ell$. By Remark \ref{rem-submersiveHyp}, we may assume that $\Lambda_2$ is $(m-1,m)$-submersive. By Lemma \ref{lem-glueAlongLine}, we can find some $Y' \subset \PP^{2N}$ with a parameterized line $\ell'$ and parameterized $N$-planes $\Lambda'_1$ and $\Lambda'_2$ such that $\Lambda_1^{'*}(Y', \ell') = (Y, \ell_1)$ and $\Lambda_2^{'*}(Y',\ell') = (Y, \ell_2)$. Because $\Lambda_1$ and $\Lambda_2$ are  $(m-1,m)$-submersive so are $\Lambda_1'$ and $\Lambda_2'$.

We can run the argument at the beginning of the proof with the pointed hypersurface $(p,Y)$ in $\PP^N$ replaced by the pointed hypersurface $(p', Y')$ in $\PP^{2N}$. We let $\phi_n': \GG_{p'}(n,2N) \dashrightarrow \cU_{n,d}$ from the space of parameterized $n$-dimensional linear spaces to degree $d$ pointed hypersurfaces in $\PP^n$. Define $B_n'$ and $\Gamma_n'$ as before. We obtain a set $Z'$ such that $\Gamma_{m-1}'$ is the family of $(m-1)$-dimensional linear spaces intersecting $Z'$.  We know $\ell'$ either intersects $Z'$ or does not intersect $Z'$. Considering $\Lambda'_1$ we conclude that $\ell'$  intersects $Z'$. On the other hand, considering $\Lambda_2'$ we conclude that $\ell'$ does not intersect $Z'$ leading to a contradiction.

This concludes the proof that $\cB_{m-1,d}$ is part of a tower of varieties induced from $\FF$. If $Z$ also parameterizes some lines lying in the hypersurface $Y$, we may repeat the above argument to conclude that $\cB_{m-1,d}$ contains all pairs $(p,X) \in \cB_{m-1,d}$ such that there is a line in $X$ passing through $p$. If $Z$ parameterizes only lines contained in $Y$, then we are in case (3). Otherwise, for all $n \geq m-1$, the closure of $\cB_{n,d}$ is also induced by $\FF$. In particular, the closure of $\cB_{m,d}$ is induced by $\FF$.

\end{proof}

We now wish to understand towers of induced varieties from families $\FF \subset \cU_{1,d}$. We first need an elementary lemma about polynomials on $\PP^1$. Recall that a vector bundle $\bigoplus_i \OO(a_i)$ on $\PP^1$ is balanced if $|a_i - a_j| \leq 1$ for all $i$ and $j$.

\begin{lemma}\label{lem-balanced}
Let $p$ be a given degree $d$ polynomial on $\PP^1$ with at least 2 distinct roots, and let $T$ be the vanishing scheme. For general polynomials $f_2, \dots, f_n$ on $\PP^1$ of degree $d-1$, we get a map $\OO(1)^{n+1} \to \OO(d)$ given by $(\partial_s p, \partial_t p, f_2, \dots, f_n)$. Then  for $d=n-1$ or $n-2$, the kernel of the map $\OO(1)^{n+1} \to T$ given by the composition $\OO(1)^{n+1} \to \OO(d) \to T(d)$ is balanced.
\end{lemma}
\begin{proof}
Since $p$ has at least 2 distinct roots, $\partial_s p$ and $\partial_t p$ have two distinct monomials $f_0$, $f_1$ with nonzero coefficients. Choose monomials $f_2, \dots, f_n$ distinct from $f_0, f_1$ such that the set $\{f_i\}_{i=0}^n$  contains $s^{d-1}, t^{d-1}$. Then up to a change of coordinates the map is given by all the monomials of degree $d$ or all but one of the monomials of degree $d$. In either case, it is easy to see that the kernel is balanced. 
\end{proof}

\begin{proposition}
\label{prop-d-1d-2cases}
Let $\FF \subset \cU_{1,d}$ be  a $\PGL_2$-invariant, closed irreducible family, and let $\cB_{n,d}$ be the tower of varieties induced by $\FF$. 
 \begin{enumerate} 
\item If $\cB_{d-1,d} \neq \cU_{d-1,d}$, then $\FF$ parameterizes the space of polynomials on $\PP^1$ with at most one root.
\item If $\cB_{d-2,d} \neq \cU_{d-2,d}$, then $\FF$ parameterizes a union of $\PGL_2$-orbits in the space of polynomials on $\PP^1$ with at most two distinct roots.  
\end{enumerate}
\end{proposition}
\begin{proof}
Since $\FF$  is a union of $\PGL_2$-orbits, it suffices to study a single $\PGL_2$-orbit. Let $a$ be the dimension of the $PGL_2$-stabilizer of a general element of $\FF$, so that $\dim \FF = 3-a$. We claim that except for the situations specified in the proposition, $\cB_{n,d} = \cU_{n,d}$ for $n = d-1, d-2$. Let $N = \binom{n+d}{d}-1$. Consider the incidence correspondence 
$$I = \{ (p,\ell, X) | p \in X, p \in \ell, \ell^* (p,X) \in \FF\} \subset \PP^n \times \GG_p(1,n) \times \PP^{N}$$
parameterizing triples of a point $p$, a parameterized line $\ell$ through $p$ and a hypersurface of degree $d$ in $\PP^n$ such that the pullback of the pointed hypersurface $(p,X)$ by $\ell$ is in the family $\FF$. Then $$\dim I = N -d +2n-2 + \dim \FF,$$ since for each element of $\FF$, there is a  $(2n-2)$-dimensional family of lines in $\PP^n$ it could be identified with, and for each such line, there is an $(N-d)$-dimensional family of hypersurfaces intersecting the given line as prescribed. Consider the map $\alpha: I \to \cU_{n,d}$ defined by projection to the first and third coordinates. We claim that for $n=d-1$ or $n=d-2$, $\alpha$ dominates $\cU_{n,d}$ unless we are in one of the two situations described in the proposition. To see that $\alpha$ dominates, it suffices to consider the dimension of a general fiber. We start with the case $n=d-1$.

Consider the pullback of the standard exact sequence to the line $\ell$
$$0 \to \ell^* T_{\PP^n} (- \log X) \to \ell^* T_{\PP^n} \to \ell^* \OO_X(X) \to 0.$$ The tangent space to a fiber of $\alpha$ at a general point $(p,\ell,X)$ of $I$ is the space of sections of $T_{\PP^n}(-\log X)$ whose image in $T_{\PP^n}$ vanish at $p$. This is given by the global sections of the sheaf $G$, where $G$ is the kernel of the map  $$\ell^* T_{\PP^n} \to \ell^* (\OO_X(X) \oplus T_{\PP^n}|_p).$$ Alternatively,  $G$ is an elementary up modification given by  $$0 \to \ell^* T_{\PP^n}(-\log X)(-p)) \to G \to \OO_p \to 0.$$ 

Thus, we need to calculate the splitting type of $\ell^* T_{\PP^n}(-\log X)$. If we can show that $\ell^* T_{\PP^n}(-\log X)$ is globally generated, then curves of that deformation class will certainly sweep out $X$, and the result on $\alpha$ will follow.

Consider a fixed polynomial $f_0$ on $\PP^1$ with at least two distinct roots. We wish to analyze when a general hypersurface $X \subset \PP^n$ with $V(f_0)$ as a linear section satisfies that lines with intersection $V(f_0)$ sweep out $\PP^n$. Consider the following diagram.

$$ \xymatrix{  & 0 & 0 & & \\ 0 \ar[r] & T_{\PP^n} (- \log X) \ar[u]\ar[r]  & T_{\PP^n}  \ar[u] \ar[r]^{\phi} &  \OO_X(X) \ar[r] & 0  \\  0 \ar[r]  & E \ar[u] \ar[r] & \OO_{\PP^n}(1)^{n+1} \ar[u]\ar[r]^{\psi} & \OO_X(X) \ar[u]_= \ar[r] & 0 \\ & \OO \ar[u] \ar[r]_= & \OO \ar[u] & \\ & 0 \ar[u] & 0 \ar[u] & }, $$

The map $\psi$ is given by the partial derivatives of the equation for $X$. Let $X$ be general among those having $V(f_0)$ as a linear section. Then since $f_0$ has at least two distinct roots, by Lemma \ref{lem-balanced}, the kernel $E$  restricts to the line $\ell$ as a balanced vector bundle. Since $n=d-1$, it follows that $E|_\ell$ and hence, $T_{\PP^n} (- \log X)|{\ell}$ must consist only of $\OO$ factors, as required.

Now suppose that $n=d-2$ and that $V(f_0)$ consists of at least three distinct points. Then by the calculation of the $n=d-1$ case, a general point of a general hypersurface $X$ in $\PP^{d-1}$ has a 1-parameter family of line sections through $p$ that are in $\FF$. If $X'$ is a general hyperplane section of $X$, $X$ has finitely many parameterized lines through a general point whose linear section is in $\FF$. The result follows.
\end{proof}

\section{Understanding $Z_1$ and $Z_2$}
\label{sec-Z1Z2}
Recall that $Z_i$ is the locus of points on a hypersurface $X$ swept out by lines that intersect $X$ in at most $i$ points. In this section we study the algebraic hyperbolicity of $Z_1$ and $Z_2$. We first introduce certain incidence correspondences and then study their images in $X$.

Let $\lambda$ be a partition of $d$ into $i$ parts. Consider the incidence correspondence $$I_{\lambda,d,n} = \{ (\ell, p_1, \dots, p_i, X) | \ell \cap X = \sum_{j=1}^i \lambda_j p_j \}$$ parameterizing lines $\ell$ with $i$ distinct marked points and hypersurfaces of degree $d$ whose restriction to the line $\ell$ has multiplicities $\lambda_j$ at the point $p_j$.

\begin{lemma}
\label{lem-irrIlambda}
$I_{\lambda,d,n}$ is irreducible of dimension $2n-3+i+\binom{n+d}{d}-d$. It follows that if $d > 2n-2+i $, then the $Z_i$ is empty for a general hypersurface. If $d > n-1+i$, then $Z_i$ is not equal to $X$ for the general hypersurface. 
\end{lemma}
\begin{proof}
The incidence correspondence $I_{\lambda,d,n}$ maps to the space of lines with $i$ marked points. The fibers are linear spaces of codimension $d$ in the space of polynomials of degree $d$ in $n+1$ variables. We conclude that $I_{\lambda,d,n}$ is irreducible of dimension $2n-3+i + \binom{d+n}{d} - d$. 
\end{proof}

We are particularly interested in $I_{\lambda}$ where $\lambda$ is the singleton partition $(d)$ or the two-part partition $(r,d-r)$. Following Clemens and Ran \cite{ClemensRan}, we introduce the osculation varieties. Let $V$ be the space of nonzero homogeneous forms of degree $d$ in $\PP^n$.  Let $$\Delta_r := \{ (\ell, x, F) | \ell \cap X \geq rx \} \subset \bG(1, n) \times \PP^n \times V$$ be the $r$th osculation variety parameterizing lines that have contact of order at least $r$ to a pointed hypersurface at the marked point. By Lemma \ref{lem-irrIlambda}, $\Delta_r$ is irreducible of dimension $2n-1-r+\binom{n+d}{d}$. Given a hypersurface $X= V(F)$, let $\Delta_r(X)$ be the fiber of $\Delta_r$ over $F \in V$.

We would also like to study a variant of $\Delta_r$. We introduce the double osculation varieties
 $$\Delta_{r,s} := \{ (\ell, x, y, F) | \ell \cap X \geq rx +sy \} \subset \bG(1, n) \times \PP^n \times \PP^n \times V$$
 parameterizing lines that have contact of order $r$ and $s$ at two points of the hypersurface. When the points coincide, we require that the contact order be $r+s$. As above, for $X = V(F)$ let $\Delta_{r,s}(X)$ be the fiber of $\Delta_{r,s}$ over $F \in V$. 
 
We now construct $\Delta_{r,s}$ explicitly and compute its canonical bundle.
 
\begin{theorem}\label{thm-generaltype}
The varieties $\Delta_{r}$ and $\Delta_{r,s}$ are smooth and irreducible. For a general hypersurface $X$, $\Delta_d(X)$ and $\Delta_{r,d-r}(X)$ are of general type when $d \geq 2\sqrt{n}+1$. Moreover, for $d \geq 2 \sqrt{n} + 1$ and any family of irreducible curves $C$ that sweep out a Zariski open set in $\Delta_d$ or $\Delta_{r,d-r}$, we have $$2g(C) -2 \geq (d-2) \ C \cdot H,$$ where $H$ is the hyperplane class in $\PP^n$. 
\end{theorem}
\begin{proof}
Smoothness and irreducibility comes from the realization of $\Delta_r$ and $\Delta_{r,s}$ as bundles over the space of pointed lines in $\PP^n$. To prove the desired results about the canonical bundle, we explicitly construct the spaces as vanishing loci of sections of vector bundles and repeatedly apply adjunction.

Let $\pi_i$ denote the projection from $\bG(1,n) \times \PP^n \times V$ or $\bG(1, n) \times \PP^n \times \PP^n \times V$ to the $i$-th factor. Forgetting the point $y$, defines a projection $\varphi: \Delta_{r,s} \to \Delta_r$. We write $(a_1,a_2)$ for the bundle $\pi_1^* \OO(a_1) \otimes \pi_2^* \OO(a_2)$ on $\bG(1, n) \times \PP^n \times V$ and $(a_1,a_2,a_3)$ for the bundle $\pi_1^* \OO(a_1) \otimes \pi_2^* \OO(a_2) \otimes \pi_3^* \OO(a_3)$ on $\bG(1, n) \times \PP^n \times \PP^n \times V$.
 
Recall the computation of the canonical bundle of $\Delta_r$ from \cite{ClemensRan}. There is a tautological bundle $\OO_V(-1)$ that associates to each point of $V$ given by a degree $d$ polynomial $F$, the $1$-dimensional subspace spanned by $F$. Note that this bundle is canonically trivial since it has a nowhere vanishing section. The pullback $\pi_3^* \OO_V(-1)$ maps to  $\pi_2^* \OO_{\PP^n}(d)$ by evaluation $$e: \pi_3^* \OO_V(-1) \to \pi_2^* \OO_{\PP^n}(d).$$ 
 Similarly, let $$0 \to S \to \OO_{\bG(1,n)}^{n+1} \to Q \to 0$$ denote the tautological sequence on $\bG(1,n)$. There is a natural map
 $$f: \pi_2^* \OO_{\PP^n}(-1) \to \pi_1^* Q$$ given by the composition of the natural map $\pi_2^* \OO_{\PP^n}(-1) \to \OO^{n+1}$ with the projection onto $Q$.

We now start applying adjunction. First, $\Delta_1$  is the common zero locus of the two maps $e$ and $f$. Hence, by adjunction we see that
 $$\omega_{\Delta_1} = \omega_{\bG(1,n)\times \PP^n \times V} \otimes \pi_1^* \det(Q) \otimes \pi_2^* \OO_{\PP^n}(d+n-1) = \OO_{\Delta_1} (-n, d-2)  $$

On $\Delta_1$, the map $\pi_2^*\OO(-1) \to \OO^{n+1}$ factors through $S$, so there is an exact sequence
$$0 \to \pi_2^* \OO_{\PP^n}(-1) \to \pi_1^* S \to R^{\vee} \to 0, $$ where $R^{\vee}$ has rank $1$. Hence $$R = \OO_{\bG(1,n)\times \PP^n \times V}(1,-1).$$

There is a filtration on $\pi_1^* \Sym^d S^{\vee}$ given by polynomials of degree $d$ vanishing to order $i$ at $x$ modulo those that vanish to order $i+1$ at $x$. 
$$\frac{F^i}{F^{i+1}} = \pi_2^* \OO_{\PP^n}(d-i) \otimes R^i = \OO_{\bG(1,n)\times \PP^n \times V}(i, d-2i).$$

$\Delta_r$ is the zero scheme of the natural map $$\OO_V(-1) \to \frac{F^1}{F^r}.$$
Hence, by adjunction
$$\omega_{\Delta_r} = \OO_{\Delta_r} \left(\frac{r(r-1)}{2}-n, rd- r(r-1) -2\right).$$
In particular, if $d \geq \sqrt{2n}+1$, then $\Delta_d$ is of general type.

We can work out the class of $\Delta_{r,s}$ similarly, using adjunction on $\Delta_1 \times \PP^n \subset \bG(1,n)\times \PP^n \times \PP^n \times V$. We have a map $$\pi_3^* \OO_{\PP^n}(-1) \to \pi_1^* Q$$ and $$\OO_V(-1) \to \pi_3^* \OO_{\PP^n}(d).$$ The locus $\Delta_{r,1}$ is defined by the common vanishing of this locus on $\Delta_r$. Hence, we have 
$$\omega_{\Delta_{r,1}} = \OO_{\Delta_{r,1}} \left(\frac{r(r-1)}{2}-n+1, rd- r(r-1) -2, d-2\right).$$

On $\Delta_r$ there is a map $\OO_V(-1) \to F^r$. We have a filtration on $R^r \otimes \Sym^{d-r} S^{\vee}$ given by the vanishing order of the polynomials at $y$.
$$\frac{G^i}{G^{i+1}} = \OO (i, 0, d-2i)$$
$\Delta_{r,s}$ is the zero locus of the map to $G^1/G^s$.

So the canonical of $\Delta_{r,s}$ is given by 
$$\omega_{\Delta_{r,s}} = \left(\frac{r(r-1)}{2}+\frac{s(s-1)}{2}-n, r(d- r+1) -2, s(d - s+1) -2\right).$$
In particular, if $d \geq 2 \sqrt{n}+1$, then $\Delta_{r, d-r}$ is of general type for all $1 \leq r \leq d$. 

If deformations of $C$ sweep out $\Delta_d$, respectively $\Delta_{r, d-r}$, then the normal bundle $N_{C/\Delta_d}$, respectively  $N_{C/\Delta_{r,d-r}}$, is globally generated hence has nonnegative degree. By the standard exact sequence for the normal bundle we conclude that
$$2g(C) -2 \geq \omega_{\Delta_d} \cdot C \quad (\mbox{resp.} \ 2g(C) -2 \geq \omega_{\Delta_{r,d-r}} \cdot C).$$ The final statement follows. 
\end{proof}


There is a natural map $\alpha: \Delta_d(X) \to X$ whose image is $Z_1$ and natural maps $\beta_k: \Delta_{k,d-k}(X) \to X$ for each $1 \leq k \leq d-1$ whose images combine to form $Z_2$. In terms of the incidence correspondences defining $\Delta_d$ and $\Delta_{r, d-r}$, these maps are induced by $\pi_2$, the projection to the second factor. 

\begin{proposition}\label{prop-injective}
For $d \geq \frac{3n-1}{2}$, $\alpha$ is injective away from $Z_L$, the locus of lines on $X$ swept out by lines. For $d \geq \frac{3n+1}{2}$, the $\beta_k$ are injective away from $Z_L$.
\end{proposition}
\begin{proof}
We begin by considering $\alpha$. Consider the locus in $\DD \subset \cU_{n,d}$ consisting of the points $(p,X)$ such that there are two lines $\ell_1$ and $\ell_2$ meeting $X$ in $d[p]$. We claim that the codimension of $\DD$ is at least $n$ in $\cU_{n,d}$, which  suffices to prove the claim.

Let $I$ be the set of tuples $(p,X,\ell_1,\ell_2)$ such that $\ell_1$ and $\ell_2$ both meet $X$ in $d[p]$. Then $\DD$ is naturally the image of $I$, so it remains to compute the dimension of $I$. We see that $I$ naturally projects onto the space $T$ of tuples $(p,\ell_1,\ell_2)$ where $p = \ell_1 \cap \ell_2$ and $\ell_1 \neq \ell_2$. We know that $T$ has dimension $2n-2+1+n-1 = 3n-2$. Then the fibers of $I$ over $T$  have dimension $N - 2d$. It follows that $I$ has dimension $N+3n-2-2d$, and so $\DD$ has at most this dimension. Since the dimension of $\cU_{n,d}$ is $N+n-1$, we see that $\DD$ has codimension at least $N+n-1-(N+3n-2-2d) = 2d-2n+1$. This is at least $n$ provided that $d \geq \frac{3n-1}{2}$.

The argument for $\beta_k$ is similar. Let $\DD_k \subset \cU_{n,d}$ be the locus such that there are two lines $\ell_1$ and $\ell_2$ meeting $X$ in $k[p]+(d-k)[q]$ for some other point $q$. Let $I_k$ be the set of tuples $(p,X,\ell_1,\ell_2, q_1, q_2)$ such that $\ell_i$ meets $X$ in $k[p]+(d-k)[q_i]$. Then a dimension count similar to the above gives that $I_k$ has dimension $N+3n-2d$, so that $\DD_k$ has codimension at least $2d-2n-1$. This is at least $n$ when $d \geq \frac{3n+1}{2}$.
\end{proof}

\begin{theorem}
\label{thm-algHypZ1}
For $d \geq \frac{3n-1}{2}$, $Z_1$ is algebraically hyperbolic outside of $Z_L$. For $d \geq \frac{3n+1}{2}$, $Z_2$  is algebraically hyperbolic outside of $Z_L$.
\end{theorem}
\begin{proof}
We can adapt an argument from \cite{RiedlYang}. Consider the universal variety $\cU'_{n,d}$ of lines $\ell$ meeting a hypersurface $X$ to order $d$ at a point, and let $\cB_{n,d}$ be the set of $(\ell,X)$ such that there is a curve in $\Delta_d(X)$ passing through $\ell$ that violates $2g-2 \geq  H \cdot C$. By Theorem \ref{thm-generaltype}, $\cB_{d,d}$ is a countable union of subvarieties of codimension at least 1.

We claim that each component of $\cB_{d-c,d}$ is codimension at least $c+1$. Let $(\ell_1, X_1)$ be a point of $\cU_{d,d} \setminus \cB_{d,d}$. Let $(\ell_0, X_0)$ be a general point of a component of $\cB_{d-c,d}$. By Lemma \ref{lem-glueAlongLine}, there is a pair $(\ell, Y) \in \cU_{N,d}$ such that $(\ell_0, X_0)$ and $(\ell_1, X_1)$ are both parameterized linear sections of $(\ell, Y)$. Let $\GG_{\ell}$ be the space of parameterized $(d-c)$-planes in $\PP^N$ containing $\ell$. This gives a map from $\GG_{\ell} \to \cU_{d-c,d}$ given by taking linear sections. By Lemma \ref{lem-cluster},  $\cB_{d-c,d}$ has codimension at least $c+1$ near $(\ell_0,X_0)$.

The same argument applies to $\Delta_{r,d-r}$. Consider the  universal variety $\cU'_{n,d}$ of lines $\ell$ meeting a hypersurface $X$ to order $r$ at a point and $d-r$ at another point, and let $\cB_{n,d}$ be the set of $(\ell,X)$ such that there is a curve in $\Delta_{r,d-r}(X)$ passing through $\ell$ that violates $2g-2 \geq  H \cdot C$. We see that $\cB_{d,d}$ is a countable union of subvarieties of codimension at least 1. Hence, by the same argument as above $\cB_{d-c,d}$ has codimension at least $c+1$ near $(\ell_0,X_0)$.

Note that $\Delta_{d}(X)$ has dimension $2n-1-d$. If $n=d-c$, $\cB_{d-c,d}$ does not intersect $\Delta_d$ for a general hypersurface if 
\[ c \geq 2n-1-d ,\]
i.e.
\[ d-n \geq 2n-1-d . \]
This is equivalent to
\[ d \geq \frac{3n-1}{2}. \]
Similarly, $\Delta_{r,d-r}(X)$ has dimension $2n-d$. If $n=d-c$, then $\cB_{d-c,d}$ does not intersect $\Delta_{r, d-r}$ for a general hypersurface if 
$$c \geq 2n-d.$$ In other words, if 
$$d \geq \frac{3n}{2}.$$
Combining with Proposition \ref{prop-injective}, the theorem follows.
\end{proof}

\section{Applications}

In this section we combine the results of Sections \ref{sec-Main} and \ref{sec-Z1Z2} to prove hyperbolicity-type results on very general hypersurfaces. Recall that $Z_i$ is the locus  on the hypersurface $X$ swept out by lines that meet $X$ in at most $i$ points and $Z_L$ is the locus on $X$ swept out by lines.

\begin{theorem}\label{thm-AH}
Let $X$ be a very general hypersurface in $\PP^n$ with $d \geq \frac{3n+2}{2}$. Then any curve not lying in $Z_L$ satisfies $2g-2 \geq H \cdot C$, where $g$ is the geometric genus of $C$. In particular, $X$ is algebraically hyperbolic outside of $Z_L$.
\end{theorem}
\begin{proof}
Let $\cB_{n,d}$ be a component of the set of pairs $(p,X)$ such that $p$ has a geometric genus $g$ curve passing through it of degree greater than $2g-2$. For $r >n$, let $\cB_{r,d}$ be the tower of varieties induced by $\cB_{n,d}$. We claim that $\cB_{d-2,d}$ is codimension at least one in $\cU_{d-2,d}$. Otherwise, we could find a family of curves that sweep out $X$ of degree greater than $2g-2$. Thus, the normal bundle of a general such curve $f: C \to X$ is globally generated, so
\[ 2g-2 - K_X \cdot C = \deg N_{f/X} \geq 0.  \]
Since $K_X = H$, it follows that $2g-2 \geq H \cdot C$, a contradiction.

By Theorem \ref{thm-MainIncidenceThm}, we see that either $\cB_{d-2-c,d}$ has codimension at least $2c+1$ in $\cU_{m-c,d}$ or there is some $\FF$ as in the statement. If $\cB_{d-2-c,d}$ has codimension at least $2c+1$, then for $d-2-c = n$, we have $2c+1 = 2d-3-2n \geq n-1$, so a very general hypersurface of degree $d$ in $\PP^n$ contains no such curves. Otherwise, $\cB_{d-2,d}$ consists of the points swept out by $\FF$. By Proposition \ref{prop-d-1d-2cases}, it follows that $\cB_{d-2,d}$ must be the locus of points in $Z_2$ as $X$ varies. By Theorem \ref{thm-algHypZ1}, $Z_2$ is algebraically hyperbolic outside of $Z_L$ and the required inequality holds for curves not contained in $Z_L$. The result follows.
\end{proof}

\begin{theorem}\label{thm-Chow0}
 Let  $X$ be a very general hypersurface in $\PP^n$ of degree $d$.  
\begin{enumerate}
\item Let $k$ be a positive integer. If $d \geq \frac{3n+1-k}{2}$, then the only points of $X$ rationally equivalent to a $k$-dimensional family of points other than themselves are those that lie in $Z_1$. 
\item If $d \geq \frac{3n}{2}$, then $X$ contains lines but no other rational curves. 
\item If $d \geq \frac{3n+3}{2}$, then any point on  $X$ rationally equivalent to another point of $X$ lies in $Z_2$.
\end{enumerate}

\end{theorem}
\begin{proof}
Let $\cB_{n,d}$ be a component of the space of pairs $(p,X)$ such that $p$ is rationally equivalent to at least a $1$-dimensional family of other points of $X$. For $r >n$, let $\cB_{r,d}$ be the tower of varieties induced by $\cB_{n,d}$. By a result of Roitman \cite{Roitman}, a very general point of a Calabi-Yau variety is rationally equivalent to only finitely many others. Thus, each component of $\cB_{d-1,d}$ is codimension at least 1 in $\cU_{d-1,d}$. By Theorem \ref{thm-MainIncidenceThm} we either have a family $\FF$ inducing $\cU_{1,d}$, or we see that $\cB_{d-1-c,d}$ has codimension at least $2c+1$ in $\cU_{d-1-c,d}$. If $\cB_{d-1,d}$ is in the tower of varieties induced by $\FF$, then by Proposition \ref{prop-d-1d-2cases}, we see that $\cB_{d-1,d}$ must be equal to the space of lines meeting the hypersurface to order $d$ at a point or the space of lines contained in the hypersurface. In this case, all the points rationally equivalent to a $k$-dimensional family of points other than themselves lie in $Z_1$.

Otherwise, $\cB_{d-1-c,d}$ has codimension at least $2c+1$ in $\cU_{d-1-c,d}$. Thus, a general hypersurface in $\PP^{d-1-c}$  contains no such points provided that $2c+1 \geq d-1-c-k$. If we set $n=d-1-c$, this holds if
\[ 2(d-1-n)+1 \geq n-k \]
or
\[ d \geq \frac{3n+1-k}{2} . \]

Since all the points on a rational curve are rationally equivalent, by setting $k=1$ in (1), we see that if $d \geq \frac{3n}{2}$,  the only rational curves on $X$ are  contained in $Z_1$. By Theorem \ref{thm-algHypZ1}, $Z_1$ is algebraically hyperbolic outside $Z_L$, so these rational curves must be contained in $Z_L$. By a result of Beheshti and Riedl \cite{BeheshtiRiedl}, if $d\geq n$  the locus of lines does not contain any rational curves other than the lines. Hence, for $d \geq \frac{3n}{2}$, a very general hypersurface contains lines but no other rational curves.

To see the last claim, we let $\cB_{r,d}$ be a component of the space of pairs $(p,X)$ such that $p$ is rationally equivalent to some other point of $X$. By Roitman's Theorem \cite{Roitman}, a very general point of a Calabi-Yau hypersurface $X$ is rationally equivalent to no others. Taking a general hyperplane section of $X$, we see that a very  general point of the hyperplane section is rationally equivalent to no other points. Thus, $\cB_{d-2,d}$ has codimension at least 1 in $\cU_{d-2,d}$. Using Theorem \ref{thm-MainIncidenceThm}, we see that either there exists an $\FF$ inducing $\cB_{d-2,d}$ or $\cB_{d-2-c,d}$ has codimension at least $2c+1$ in $\cU_{d-2-c,d}$. If there exists an $\FF$, then we know that points of $X$ rationally equivalent to another lie in $Z_2$ by Proposition \ref{prop-d-1d-2cases}. Otherwise, a general hypersurface in $\PP^{d-2-c}$  contains no points of $\cB_{d-2-c,d}$ if $2c+1 \geq d-2-c$. Setting $n = d-2-c$, we see that this happens if
\[ 2(d-n-2) + 1 \geq n  \]
or
\[ d \geq \frac{3n+3}{2} .\]

\end{proof}

Our last result relies on a version of the Green-Griffiths-Lang Conjecture. Recall that this conjecture predicts that given a smooth hypersurface $X$ with $d \geq n+2$, there is some exceptional locus $Z \subset X$ containing the images of all entire curves. If we assume this conjecture and we assume furthermore that these exceptional loci form an algebraic family in the universal hypersurface, we can identify precisely what the exceptional locus should be for $d \geq \frac{3n+2}{2}$.

\begin{theorem}\label{thm-GGL}
Suppose for each $n$ there is a countable union of varieties $\cB_{n,d} \subset \cU_{n,d}$ such that any entire curve in a fiber of the map $\cU_{n,d} \to \PP H^0(\PP^n, \OO_{\PP^n}(d))$ is contained in $\cB_{n,d}$. Suppose that $\cB_{d-2,d}$ is not equal to $\cU_{d-2,d}$. Then for a very general hypersurface with $d \geq \frac{3n+2}{2}$, any entire curve is contained in $Z_2$.
\end{theorem}
\begin{proof}
By Theorem \ref{thm-MainIncidenceThm} and Proposition \ref{prop-d-1d-2cases}, we see that for $d \geq \frac{3n+2}{2}$, any point of $\cB_{n,d}$ is either contained in the locus swept out by lines meeting the hypersurface in two distinct points or has codimension larger than $n$. The result follows.
\end{proof}

\bibliographystyle{plain}

\end{document}